\documentclass[12pt]{amsart}

\usepackage{cancel}
\usepackage{latexsym, amssymb, amsmath}
\usepackage{soul}
\usepackage{amsfonts, graphicx}
\usepackage{graphicx,color}
\newcommand{\be}{\begin{equation}}
\newcommand{\ee}{\end{equation}}
\newcommand{\beq}{\begin{eqnarray}}
\newcommand{\eeq}{\end{eqnarray}}

\newtheorem{thm}{Theorem}[section]

\newtheorem{lma}{Lemma}[section]
\newtheorem{prop}{Proposition}[section]
\newtheorem{cor}{Corollary}[section]

\theoremstyle{remark}
\newtheorem{rem}{Remark}[section]
\numberwithin{equation}{section}

\newtheorem{ex}{Example}[section]

\def\be{\begin{equation}}
\def\ee{\end{equation}}
\def\bee{\begin{equation*}}
\def\eee{\end{equation*}}

\def\lf{\left}
\def\ri{\right}

\def\K{K\"ahler }

\def\KE{K\"ahler-Einstein }
\def\KR{K\"ahler-Ricci }
\def\Ric{\text{\rm Ric}}
\def\Rm{\text{\rm Rm}}

\def\ddb{\sqrt{-1}\partial\bar\partial}
\def\p{\partial}

\def\heat{\lf(\frac{\p}{\p t}-\Delta_t\ri)}
\def\tr{\operatorname{tr}}

\def\e{\epsilon}

\begin{document}

\title[]
{Complete K\"ahler-Einstein metric on Stein manifolds with negative curvature}

 \author{Man-Chun Lee}
\address[Man-Chun Lee]{Mathematics Department, Northwestern University, 2033 Sheridan Road, Evanston, IL 60208.}
\email{mclee@math.northwestern.edu}

\thanks{The author was partially supported by NSF grant DMS-1709894.}

\renewcommand{\subjclassname}{
  \textup{2010} Mathematics Subject Classification}

\date{\today}

\begin{abstract}
We show the existence of complete negative \KE metric on Stein manifolds with negatively pinched holomorphic sectional curvature. We prove that any \K metrics on such manifolds can be deformed to the complete negative \KE metric using the normalized \KR flow.
\end{abstract}

%\keywords{\KR flow, \K Einstein metric}

\maketitle

\markboth{Man-Chun Lee}{Complete K\"ahler-Einstein metric on Stein manifolds with negative curvature}
\section{introduction}

In \cite{WuYau2017}, Wu and Yau proved that if a complete noncompact \K manifold supports a complete bounded curvature \K metric with holomorphic sectional curvature bounded from above by a negative constant, then it supports a complete negative \KE metric with bounded curvature. This extended the previous work \cite{WuYau2016,TosattiYang2015} to the noncompact case. Using the \KR flow approach, this was recovered by Tong \cite{Tong2018}, see also the compact case in \cite{Nomura2017}. A natural question is to ask if the curvature boundedness assumption is necessary in order to obtain the existence of \KE without bounded curvature conclusion. In \cite{HuangLeeTamTong2018}, the author together with the collaborators showed that the curvature boundedness assumption can be weakened to the existence of a exhaustion function with bounded complex hessian. The main ingredient is the construction of long-time \KR flow solution using Chern-Ricci flow approximation technique introduced by the author and Tam in \cite{LeeTam2017}.

In this note, we continue to study the existence problem using the \KR flow approach. We focus on the special case when $M$ is a Stein manifold. We show the existence of complete negative \KE on $M$ by establishing a long-time solution to the \KR flow on $M$. More precisely, we showed that on the negatively curved Stein manifold $M$, the \KR flow has a long-time solution starting from \textit{any} \K metric which can be incomplete and have unbounded curvature.

\begin{thm}\label{main}
Suppose $M^n$ is a Stein manifold and $h$ is a complete \K metric on $M$ with holomorphic sectional curvature $H_h\leq -\kappa$ for some $\kappa>0$. Then any \K metric $g_0$ on $M$ admits a longtime solution to the \KR flow on $M\times [0,+\infty)$ such that $g(t)$ is instantaneous complete for $t>0$. 
\end{thm}

On general \K manifold, the short-time existence of the \KR flow was first studied by Shi \cite{Shi1989,Shi1997} when the initial metric $g_0$ is complete with bounded curvature. In general without curvature condition, the existence theory was still being unclear except the surface case. In the case of surface, the existence of the Ricci flow starting from incomplete metric has been studied in details by Giesen and Topping \cite{GiesenTopping2010,GiesenTopping2011,GiesenTopping2013}. This can be viewed as a partial generalization of their result in higher dimension.

By combining the existence with the convergence result in \cite{HuangLeeTamTong2018}, we have the existence of the complete negative \KE metric.
\begin{cor}\label{main-cor}
Suppose $M$ is a Stein manifold and $h$ is a complete \K metric on $M$ with holomorphic sectional curvature $H_h\leq -\kappa$ for some $\kappa>0$, then $M$ admits a unique complete \KE metric with negative scalar curvature.
\end{cor}

Using Corollary \ref{main-cor}, we know that the following situation admit complete negative \KE metric.
\begin{ex}Using a result of Wu \cite{Wu1967} and Corollary \ref{main-cor}, any complete simply connected \K manifold $(M,g)$ with sectional curvature bounded from above by $-1$ admits a complete negative \KE metric.
\end{ex}

\begin{ex}
Let $\mathbb{B}^{n-1}$ be the unit ball in $\mathbb{C}^n$, then any  proper embedded complex sub-manifold $\Sigma$ of $\mathbb{B}^{n-1}$ admits a complete negative \KE metric. To see this, $\mathbb{B}^{n-1}$ admits the Bergman metric $h$ with constant holomorphic sectional curvature. By the decreasing property of holomorphic curvature and properness, the pull-back metric of $h$ on $\Sigma$ is a complete \K metric with holomorphic sectional curvature bounded from above by negative value. Then the existence followed from Corollary \ref{main-cor}.
\end{ex}

\begin{ex}
Let $\mathbb{D}$ be the unit disk in $\mathbb{C}$ associated with the Poincar\'e metric. Let $\Sigma$ be a proper embedded complex sub-manifold $\Sigma$ of $\mathbb{D}^n=\mathbb{D}\times \cdots\times \mathbb{D}$. The produce metric on $\mathbb{D}^n$ is complete with holomorphic sectional curvature bounded from above by negative value. Then $\Sigma$ admits a complete negative \KE metric using the argument in previous example.
\end{ex}

\section{a-priori estimates for the \KR flow}
Let $M^n$ be a complex manifold. The \KR flow on $M$ starting from initial metric $g_0$ is a family of \K metric $g(t)$ which satisfies 
\begin{equation}
\left\{
\begin{array}{ll}
\displaystyle \frac{\partial}{\partial t} g_{i\bar j}&=-R_{i\bar j} ;\\
g(0)&=g_0.
\end{array}
\right.
\end{equation}
on $M\times [0,T]$ for some $T>0$. We say that $g(t)$ is a complete solution if $g(t)$ is a complete metric on $M$ for each $t\in [0,T]$.

\begin{lma}\label{ins-low}
Let $g(t)$ be a complete solution to the \KR flow on $\Omega\times [0,T]$ for $T\leq +\infty$. Suppose $h$ is a \K metric on $\Omega$ which can possibly be incomplete such that its holomorphic sectional curvature is bounded from above by $-\kappa$ for some $\kappa>0$. Then on $\Omega \times (0,T]$, $g(t)$ satisfies
$$\frac{(n+1)\kappa t}{2n}h\leq g(t).$$
\end{lma}
\begin{proof}
Using parabolic Schwarz Lemma gives
\begin{equation}
\heat  \tr_gh \leq g^{i\bar j}g^{k\bar l}\tilde R_{i\bar jk\bar l}.
\end{equation}
where $\tilde R$ denotes the curvature tensor of metric $h$. By using a trick of Royden \cite[Lemma, p.552]{Royden1980}, the bisectional curvature quantities is bounded from above,
\begin{equation}
g^{i\bar j}g^{k\bar l}\tilde R_{i\bar jk\bar l}\leq -\frac{n+1}{2n}\kappa (\tr_gh)^2.
\end{equation}
Noted that the inequality also holds when $h$ is not complete. Hence, we have 
\begin{equation}\label{evo-trace}
\heat  \tr_gh \leq -\frac{n+1}{2n}\kappa (\tr_gh)^2.
\end{equation}

If $T<+\infty$, applying \cite[Lemma 5.1]{HuangLeeTamTong2018} with $Q=\tr_gh$ gives 
$$\tr_g h\leq \frac{2n}{\kappa(n+1)t}$$
on $\Omega \times [0,T]$. If $T=+\infty$, then apply it on $[0,L]$ and followed by letting $L\rightarrow +\infty$.
\end{proof}

We also have a instantaneous bound on the scalar curvature. 

\begin{lma}\label{inst-R}
Let $g(t)$ be a complete solution to the \KR flow on $\Omega\times [0,T]$ for $T\leq +\infty$, then the scalar curvature satisfies 
$$R_{g(t)}\geq -\frac{n}{t}.$$
\end{lma}
\begin{proof}
Noted that the scalar curvature $R$ satisfies $\heat R\geq \frac{1}{n}R^2$. The proof is similar to the that in Lemma \ref{ins-low} by applying \cite[Lemma 5.1]{HuangLeeTamTong2018} with $Q=-R$.
\end{proof}

We also need the following version of pseudolocality of Ricci flow from \cite{HeLee2018}.
\begin{prop}\label{pseudo}
For any $m\in \mathbb{N}$ and $v>0$, there is $\sigma(m,v),\Lambda(m,v)>0$ such that if $g(t)$ is a complete Ricci flow solution on $M^m\times [0,T]$ with $\sup_{[\tau,T]}|\Rm(g(t))|<+\infty$ for all $\tau>0$ and satisfies
\begin{enumerate}
\item $|\Rm(g(0))|\leq r^{-2}$ on $B_{g_0}(p,r)$;
\item $V_{g_0}\left(B_{g_0}(p,r)\right)\geq r^m v$,
\end{enumerate}
for some $p\in M$ and $r>0$, then we have 
$$|\Rm(x,t)|\leq \Lambda r^{-2}$$
on $B_{g(t)}(p,\sigma r)$, $t \leq \min\{ T,\sigma^2 r^2\}$.
\end{prop}
\begin{proof}
The follows by applying \cite[Corollary 3.1]{HeLee2018} to the rescaled Ricci flow solution $r^{-2}g(r^2t)$ and then rescale it back.
\end{proof}

\section{approximation using pseudoconvex domain}\label{appro-stein}
In this section, we will construct local \KR flow solution using the Steinness of $M$. By a result of Grauert \cite{Grauert1958}, there exists a smooth strictly pluri-subharmonic function $\rho$ on $M$ which is a exhaustion function. For $R>1$ large, denote its sublevel set to be
$$U_R=\{x\in M: \rho(x)<R\}.$$
This is a bounded strictly pseudo-convex domain in $M$ as defined in \cite{ChengYau1982}. By Sard's Theorem, we may find a sequence of $R_m\rightarrow +\infty$ such that $d\rho \neq 0$ on $\partial U_{R_m}$. By the result in \cite{ChengYau1982}, there exists a complete \KE metric $\tilde\omega_m=-\Ric(\tilde\omega_m)$ with bounded curvature defined on $U_{R_m}$. Now we will construct approximating \KR flow using $\{(U_{R_m},\tilde\omega_m)\}_{m\in \mathbb{N}}$. Let $\omega_0$ be the \K form of the initial metric $g_0$. On $U_{R_m}$, define $\omega_{0,m,\e}=\omega_0 +\e \tilde\omega_m$. We will use $g_{0,m,\e}$ to denote the corresponding \K metric.

\begin{lma}\label{uniform-equiv-appro}
For any $\e>0$, $\omega_{0,m,\e}=\omega_0 +\e \tilde\omega_m$ is a complete \K metric on $U_{R_m}$ which is uniformly equivalent to $\tilde \omega_m$.
\end{lma}
\begin{proof}
The completeness is clear. It remains to establish the upper bound of $\omega_{0,m,\e}$ with respect to $\tilde \omega_m$. Since $\rho_{i\bar j}>0$ and $U_{R_m}$ is pre-compact in $M$, there is $C_m>>1$ such that 
\begin{equation}
\omega_0 \leq C_m \ddb \rho
\end{equation}
on $U_{R_m}$. On the other hand by \cite{ChengYau1982}, $\tilde\omega_m$ is uniformly equivalent to the standard metric $-\ddb\log (R-\rho)$ on $U_{R_m}$. Therefore,
\begin{equation}
\begin{split}
\omega_{0}&\leq C_m \ddb \rho\\
& \leq -C_m\ddb\log  (R-\rho)\\
&\leq C\tilde \omega_m.
\end{split}
\end{equation}
This completes the proof.
\end{proof}

Using the uniform equivalence of metrics, we can deform the \K metric using the \KR flow on each $U_{R_m}$ with long-time existence.
\begin{prop}\label{approx-1}
There exists a complete long-time solution $\omega_{m,\e}(t)$ to the \KR flow with initial metric $\omega_{0,m,\e}$ on $U_{R_m}\times [0,+\infty)$ such that $\omega_{m,\e}(t)$ is uniformly equivalent to $\tilde \omega_m$ on any $[a,b]\subset [0,+\infty)$.
\end{prop}
\begin{proof}
From Lemma \ref{uniform-equiv-appro}, there exists $\delta>0$ such that
\begin{equation}
\left\{
\begin{array}{ll}
\omega_{0,m,\e}\leq \delta^{-1} \tilde\omega_m;\\
\omega_{0,m,\e}-s\Ric(\tilde\omega_m)&>\delta \tilde\omega_m.\\
\end{array}
\right.
\end{equation}
for any $s>0$. The proposition follows from \cite[Theorem 1.4]{ChauLee2019}.
\end{proof}

\section{Proof of main theorem}
In this section, we will construct a global longtime solution $g(t)$ to the \KR flow with initial metric $g(0)=g_0$ using the idea in \cite{GiesenTopping2011} adapted in higher dimensional \K case. The main goal is to show that $g(t)=\lim_{m\rightarrow +\infty}\lim_{\e \rightarrow 0}g_{m,\e}(t)$ exists. In \cite{HuangLeeTam2019,ChauLee2019}, it was showed that $g_m(t)=\lim_{\e_i\rightarrow }g_{m,\e_i}(t)$ exists smoothly up to $t=0$ for some subsequence $\e_i\rightarrow 0$ using estimates from parabolic Monge-Amp\`ere equation. Since the information of $\Ric(h)$ on $M$ is missing, we need to take an alternative approach to obtain the estimates for compactness when $m\rightarrow +\infty$.

\begin{proof}[Proof of Theorem \ref{main}]
From the discussion in section \ref{appro-stein}, there is a sequence of bounded strictly pseudo-convex domain $U_{R_m}$ exhausting $M$. Moreover, there is a sequence of complete \KR flow $g_{m,\e}(t)$ defined on $U_{R_m}\times [0,+\infty)$ with $g_{m,\e}(0)=g_{0,m,\e}$. Because of completeness, we may apple Lemma \ref{ins-low} and Lemma \ref{inst-R} to deduce that for all $t\in [0,+\infty)$, 
\begin{equation}\label{inst-bdd}
\left\{
\begin{array}{ll}
&\displaystyle \frac{(n+1)\kappa t}{2n} h\leq g_{m,\e}(t);\\
&R_{g_{m,\e}(t)}\geq -nt^{-1}.
\end{array}
\right.
\end{equation}

Here we regard $h$ as a incomplete metric on the pre-compact set $U_{R_m}$. Moreover by uniform equivalence with a metric with bounded curvature, the \KR flow $g_{m,\e}(t)$ has bounded curvature when $t>0$, for example see \cite{ChauLee2019,ShermanWeinkove2012,LottZhang2016}.

Fix a compact set $\Omega\Subset M$, since $g_{0,m,\e}$ converges to $g_0$ uniformly in any $C^k$ as $\e\rightarrow 0$ and $m\rightarrow +\infty$, we can find $r>0$ and $v>0$ small enough such that for all $x\in \Omega$, $m$ large and { $\e_0(n,m,g_0,\Omega)>\e>0$}, we have 
\begin{enumerate}
\item [(a)] $B_{g_{0,m,\e}}(x,r)\Subset U_{R_m}$;
\item [(b)] $|\Rm(g_{0,m,\e})|\leq r^{-2}$ on $B_{g_{0,m,\e}}(x,r)$;
\item [(c)] $V_{g_{0,m,\e}}\left(B_{g_{0,m,\e}}(x,r)\right)\geq v r^{2n}$.
\end{enumerate}

By Proposition \ref{pseudo}, for all $m$ large and { $\e<\e_0(n,m,g_0,\Omega)$}, we have
\begin{equation}
\sup_{\Omega}|\Rm(g_{m,\e})(x,t)|\leq \Lambda(n,g_0,\Omega)
\end{equation}
for $0\leq t\leq \sigma(n,g_0,\Omega)$. In particular, for all $(x,t)\in \Omega\times  [0,\sigma]$, 
\begin{equation}\label{equiv-imm}
e^{-\Lambda \sigma} g_{0,m,\e}\leq g_{m,\e}(t)\leq e^{\Lambda \sigma} g_{0,m,\e}.
\end{equation}

For $t>\sigma$, we make use of \eqref{inst-bdd}. Using the evolution of volume form with the second inequality in \eqref{inst-bdd}, we have 
\begin{equation}
\frac{\partial}{\partial s} \left(\log \frac{\det g_{m,\e}(s)}{\det h}\right)=-R_{g_{m,\e}(t)}\leq \frac{n}{s}.
\end{equation}

By integrating it from $\sigma$ to $t$, we have 
\begin{equation}\label{vol-bound-sigma}
\begin{split}
\det g_{m,\e}(t)&\leq \left( \frac{t}{\sigma}\right)^n\det g_{m,\e}(\sigma)\\
&\leq e^{n\Lambda \sigma}\left( \frac{t}{\sigma}\right)^n \det g_{0,m,\e}.
\end{split}
\end{equation}

Combines \eqref{vol-bound-sigma}, \eqref{inst-bdd} and the elementary inequality $\tr_hg \leq \frac{\det g}{\det h} (\tr_gh)^{n-1}$, we deduce that for all $(x,t) \in \Omega\times [\sigma,+\infty)$,
\begin{equation}\label{local-unif}
\frac{(n+1)\kappa t}{2n}h\leq g_{m,\e}(t)\leq  \left( \frac{e^{\Lambda\sigma}t}{\sigma}\right)^n\left[ \frac{2n}{(n+1)\kappa t}\right]^{n-1}\left(\frac{\det g_{0,m,\e}}{\det h}\right) h.
\end{equation}

Using \eqref{equiv-imm}, \eqref{local-unif} and the fact that $g_{0,m,\e}$ converges to $g_0$ locally uniformly, in conclusion we have shown that for any $T<\infty$ and compact set $\Omega\Subset M$, there is $\lambda>1$ such that for all $m$ sufficiently large, $\e$ sufficiently small and $t\in [0,T]$, 
$$\lambda^{-1} h \leq g_{m,\e}(t)\leq \lambda h.$$

By applying higher order estimate of the \KR flow \cite{ShermanWeinkove2012} on local chart, we have all higher order estimates of $g_{m,\e}(t)$ on $\Omega \times [0,T]$ which is independent of $m>>1$ and { $\e<\e_0(n,m,g_0,\Omega)$}. 

By Ascoli-Arzel\`a Theorem, { for each $m$ large} we may first let $\e_i\rightarrow 0$ for some subsequence $\e_i$ to obtain a \KR flow solution $g_m(t), t\in [0,T]$ on each $U_{R_m}$. Then by diagonal subsequence argument and Ascoli-Arzel\`a Theorem again, we may let $m_i\rightarrow +\infty$ and followed by $T\rightarrow +\infty$ to obtain a global solution $g(t)$ of the \KR flow on $M\times [0,+\infty)$. The instantaneous completeness of $g(t)$  followed by passing \eqref{inst-bdd} to the limiting solution $g(t)$. 
\end{proof}

\begin{proof}[Proof of Corollary \ref{main-cor}]
By Theorem \ref{main}, there is a long-time solution to the \KR flow starting from $h$. The existence of complete negative \KE followed from \cite[Theorem 5.1]{HuangLeeTamTong2018}.
\end{proof}

\begin{rem}
It is clear from the proof of Theorem \ref{main} and Corollary \ref{main-cor} that the reference metric $h$ can be a incomplete metric. Then the resulting \KR flow solution and \KE metric maybe a incomplete metric.
\end{rem}

\begin{rem}
Since $h$ only served as a reference metric in the Schwarz Lemma, the existence of \K metric $h$ can be replaced by the existence of a Hermitian metric with negatively pinched \textit{real bisectional curvature}, see \cite{YangZheng2019,Tang2019}.
\end{rem}


\begin{thebibliography}{1000}
\bibitem{ChauLee2019} A. Chau; Lee, M.-C., {The \K Ricci flow around complete bounded curvature \K metrics}, arXiv:1904.02233. To appear in Trans. Amer. Math. Soc


\bibitem{ChengYau1982} Cheng, S.-Y.; Yau, S.-T.,{\sl  On the existence of a complete \K Metric on noncompact complex manifolds and the regularity of Fefferman's equation}, Communication of Pure and Applied Mathematics,  \textbf{33} (1980), no. 4, 507-544.



\bibitem{GiesenTopping2010}Giesen, G.; Topping, P.-M.,{\sl Ricci flow of negatively curved incomplete surfaces}, Calc. Var. and PDE, 38 (2010), 357-367.

\bibitem{GiesenTopping2011}Giesen, G.; Topping, P.-M.,{\sl Existence of Ricci flows of incomplete surfaces}, Comm. Partial Differential Equations, 36 (2011), 1860-1880.

\bibitem{GiesenTopping2013} Giesen, G.; Topping, P.-M.,{\sl Ricci flows with unbounded curvature}, Math. Zeit., 273 (2013), 449-460.


\bibitem{Royden1980}H.L. Royden, {\sl The Ahlfors-Schwarz lemma in several complex variables}, Comment. Math. Helv. 55 (1980), no. 4, 547–558.



\bibitem{HeLee2018} He, F.; Lee, M.-C.,{\sl Weakly PIC1 manifolds with maximal volume growth}, arXiv:1811.03318.

\bibitem{HuangLeeTam2019}Huang, S.-C.; Lee, M.-C.; Tam, L.-F., {\sl Instantaneously complete Chern-Ricci flow and K\"ahler-Einstein metrics}, Calculus of Variations and PDE(2019) 58: 161. 

\bibitem{HuangLeeTamTong2018}Huang, S.-C; Lee, M.-C.; Tam L.-F.; Tong, F.,{\sl Longtime existence of K\"ahler- Ricci flow and holomorphic sectional curvature}, arXiv:1805.12328.



\bibitem{Grauert1958}H. Grauert, {\sl On Levi’s problem and the imbedding of real analytic manifolds}, Ann. of Math. 68 (1958) 460–472, MR 0098847, Zbl 0108.07804.

\bibitem{Tang2019} K. Tang, {\sl On real bisectional curvature and K\"ahler-Ricci flow}. Proc. Amer. Math. Soc, 147 (2019), no. 2, 793-798.

\bibitem{LeeTam2017} Lee, M.-C.; Tam L.-F.,{\sl Chern-Ricci flows on noncompact complex manifolds}, To appear in J. Differential Geometry.


\bibitem{LottZhang2016}Lott, J.; Zhang, Z., {\sl Ricci flow on quasiprojective manifolds II}, J. Eur. Math. Soc.
(JEMS) 18 (2016), no. 8, 1813–1854, MR3519542, Zbl 1351.53081.


 \bibitem{Nomura2017} Nomura, R., {\sl K\"ahler Manifolds with Negative Holomorphic Sectional Curvature, K\"ahler-Ricci Flow Approach}, International Mathematics Research Notices, Vol. 2017, No. 00, pp. 16.

\bibitem{Perelman2002}Perelman, G., {\sl The entropy formula for the Ricci flow and its geometric applications}, arXiv:math.DG/0211159





\bibitem{ShermanWeinkove2012}Sherman, M.; Weinkove, B., {\sl Interior derivative estimates for the K\"ahler-Ricci flow},
Pacific J. Math. 257 (2012), no. 2, 491–501, MR2972475, Zbl 1262.53056.

\bibitem{Shi1989}Shi, W.-X., {\sl Deforming the metric on complete Riemannian manifolds}. J. Differ. Geom. 30, 223–
301 (1989).

\bibitem{Shi1997}Shi, W.-X., {\sl Ricci flow and the uniformization on complete noncompact \K manifolds}, J. Differential Geom. 45 (1997), no. 1, 94220.

\bibitem{TosattiYang2015} Tosatti, V.; Yang, X.-K., {\sl An extension of a theorem of Wu-Yau}, to appear in J. Differential Geom.,arXiv:1506.01145 .



\bibitem{Tong2018} Tong, F., {\sl The K\"ahler-Ricci flow on manifolds with negative holomorphic curvature}, arXiv:1805.03562.

\bibitem{Wu1967}H. Wu, {\sl Negatively curved K\"ahler manifolds}, Notices Amer. Math. Soc., \textbf{14} (1967), 515.

\bibitem{WuYau2017}Wu, D.-M.; Yau, S.-T., {\sl Invariant metrics on negatively pinched complete \K manifolds}, arXiv:1711.09475.

\bibitem{WuYau2016} Wu, D.-M.; Yau, S.-T., {\sl Negative Holomorphic curvature and positive canonical bundle}, Invent. Math. 204 (2016), no. 2, 595–604.


\bibitem{YangZheng2019} Yang, X.; Zheng, F., {\sl On real bisectional curvature for Hermitian manifolds}, Trans. Amer. Math. Soc., 371(4):2703–2718, 2019.

\end{thebibliography}
\end{document}